\documentclass[12pt]{article}
\RequirePackage{amsthm, amsmath, amsfonts, amssymb, graphicx}
\RequirePackage[numbers,square]{natbib}
\usepackage{dsfont}

\usepackage{float}
\usepackage[shortlabels]{enumitem}
\usepackage{booktabs}
\newcommand{\bneqn}{\vspace{-0.25cm}\begin{eqnarray}}
\newcommand{\eneqn}{\end{eqnarray}}
\allowdisplaybreaks
\usepackage[utf8]{inputenc}
\usepackage[toc,page,title,titletoc]{appendix}
\usepackage{hyperref,bookmark}
\hypersetup{
  colorlinks=true,
  linkcolor=red,
  citecolor=blue,
  filecolor=blue,
  urlcolor=blue,
}

\newtheorem{theorem}{Theorem}
\newtheorem{lemma}[theorem]{Lemma}
\newtheorem{remark}[theorem]{Remark}
\newtheorem{proposition}[theorem]{Proposition}

\advance \topmargin by -\headheight
\advance \topmargin by -\headsep
\advance \topmargin .2in
\title{The existence of the least favorable noise}
\author{ Dongzhou Huang\footnote{Department of Statistics, Colorado State University, United States of America. E-mail: dongzhou.huang@colostate.edu}}
\begin{document}
\maketitle

\begin{abstract}
Suppose that a random variable $X$ of interest is observed. This paper concerns ``the least favorable noise'' $\hat{Y}_{\epsilon}$, which maximizes the prediction error $E [X - E[X|X+Y]]^2 $ (or minimizes the variance of $E[X| X+Y]$) in the class of $Y$ with $Y$ independent of $X$ and $\mathrm{var} Y \leq \epsilon^2$. This problem was first studied by Ernst, Kagan, and Rogers (\cite{10.1214/22-ECP467}). In the present manuscript, we show that the least favorable noise $\hat{Y}_{\epsilon}$ must exist and that its variance must be $\epsilon^2$. The proof of existence relies on a convergence result we develop for variances of conditional expectations. Further, we show that the function $\inf_{\mathrm{var} Y \leq \epsilon^2} \, \mathrm{var} \, E[X|X+Y]$ is both strictly decreasing and right continuous in $\epsilon$.
\end{abstract}

\section{Introduction} \label{sec:introduction}

\hspace*{0.5cm} In 2022, Ernst, Kagan and Rogers (\cite{10.1214/22-ECP467}) investigated the problem of ``the least favorable noise'' for an observed square-integrable random variable $X$. The authors considered $\hat{Y}_{\epsilon}$, a square-integrable random variable independent of $X$, which maximizes the prediction error
\begin{equation*}
E\left[ X- E[ X| X+Y] \right]^{2} = \mathrm{var} X - \mathrm{var} \, E[X|X+Y],
\end{equation*}
(or, equivalently, minimizes the variance of $E[X| X+Y]$) in the class of $Y$ with  $\mathrm{var} Y \leq \epsilon^2$. The authors proceeded to characterize the least favorable noise and show that $Y$ should be the least favorable noise if the distribution of $Y$ satisfies the conditions of a given characterization.\\
\indent The present manuscript takes a step back from the Ernst et al. characterization of the least favorable noise and asks the following question: `does the least favorable noise exist?' In other words, given the distribution of $X$, does there exist a distribution of $Y$ which achieves the maximum of $E\left[ X- E[ X| X+Y] \right]^{2}$? Although Ernst et al. show the existence of the least favorable noise given (i) the distribution of $X$ and (ii) given that the distribution of $Y$ satisfies three characterization conditions in \cite{10.1214/22-ECP467}, the conditions are somewhat complicated and, more importantly, nearly impossible to verify for most distributions of $X$. Therefore, the question of existence of the least favorable noise has remained open. The key contribution of the present paper is to close this problem by showing the existence of the least favorable noise for any distribution of $X$. The proof relies on a convergence result for the variance of the conditional expectation of $X_n$ given $X_n + Y_n$, provided $(X_n, Y_n)$ converges weakly (see Theorem \ref{prop:keylemmatoshowexistence}).

Before proceeding with the proof of the existence of the least favorable noise, we pause to provide some practical implications of the least favorable noise. In some applied scientific scenarios, the observed random signal $X$ may be highly volatile, making it difficult to analyze. In this case, one may wish to simplify the signal $X$ while keeping its main structure. One possible way to do so is to consider the random variable
\begin{equation*}
Q=E\left[ X \big| X+ \hat{Y}_{\epsilon} \right],
\end{equation*}
where $\hat{Y}_{\epsilon}$ is the least favorable noise. The random variable $Q$ has two important properties: (i) The variance of $Q$ is no more than the variance of $X$ and achieves the minimum of $\mathrm{var}\, E[X| X+Y]$ over $Y$ satisfying $\mathrm{var} Y \leq \epsilon^2$, meaning that it is less volatile, and (ii) noting that $X+ \hat{Y}_{\epsilon}$ is close to $X$ for small $\epsilon$, the random variable $Q$, as a function of $ X+ \hat{Y}_{\epsilon} $, preserves the structure of $X$.

We now formalize the problem under consideration. To avoid trivialities, we assume $X$ is a non-degenerate and  that it is a square-integrable random variable. Without loss of generality, we may assume the mean of $X$ to be $0$. Since the conditional expectation $E[X|X+Y]$ remains the same if we shift $Y$ by a constant, we shall only focus on the following class of random variables $Y$:
\begin{equation*}
 \mathcal{V}_{\epsilon}(X) := \{ Y : Y \text{ independent of $X$},\, E[Y]=0 \text{ and } E[Y^2] \leq \epsilon^{2}  \}.
\end{equation*}
We define
\begin{equation*}
L(X, \epsilon) := \inf_{Y \in \mathcal{V}_{\epsilon}(X)} \mathrm{var}\, E[X|X+Y].
\end{equation*}
Then the least favorable noise $\hat{Y}_{\epsilon}$ is a random variable $Y$ in $\mathcal{V}_{\epsilon}(X)$ such that $\mathrm{var} \, E[X|X+Y]$ achieves $L(X, \epsilon)$.

In Section \ref{sec:existence}, we prove the existence of the least favorable noise. That is, we prove the existence of $\hat{Y}_{\epsilon}$ which minimizes $\mathrm{var}\, E[X|X+Y]$ when $\mathrm{var} Y \leq \epsilon^2$. In general, to show the existence of the minimizer of a value function $f(x)$, one typically adopts the following strategy:
\begin{enumerate}[label=(\roman*)]
\item One constructs a sequence $\{x_n\}_{n=1}^{\infty}$ such that $f(x_n)$ converges to $\inf\limits_{x} f(x)$;
\item One finds a convergent subsequence of $\{x_n\}_{n=1}^{\infty}$ which converges to $x_{*}$;
\item One then shows that $f(x_{*}) = \inf\limits_{x} f(x)$, from which one concludes that $x_{*}$ is the minimizer.
\end{enumerate}
In this paper, we indeed follow the above strategy. Firstly, we consider a sequence $\{ Y_n\}_{n=1}^{\infty}$ such that $\mathrm{var} Y_n \leq \epsilon^2$ and such that the variance of $E[X|X+Y_n]$ converges to $\inf\limits_{Y} \mathrm{var}\, E[X|X+Y]$. Secondly, by the tightness of $\{ Y_n\}_{n=1}^{\infty}$, there exists a subsequence that converges weakly to a random variable, say, $\tilde{Y}$. It then remains to show that the variance of $E[X| X + \tilde{Y}]$ is exactly $\inf\limits_{Y} \mathrm{var}\, E[X|X+Y]$. This represents the most mathematically challenging task of this work. To this end, we develop a convergence result for the variances of conditional expectations (see Theorem \ref{prop:keylemmatoshowexistence}), which is in part inspired by the work of \cite{crimaldi2005convergence}.

The remainder of this paper is organized as follows. In Section \ref{sec:property}, we prove that the variance of the least favorable noise must be $\epsilon^2$, allowing us to reduce the class of $Y$ when searching for the least favorable noise. Indeed, this result simplifies the three characterization conditions in \cite{10.1214/22-ECP467} to two characterization conditions. In Section \ref{sec:LXepsilon}, we consider further properties of $L(X, \epsilon)$. We prove that it is both strictly decreasing and right continuous in $\epsilon$ on $[0, \infty)$. Consequently, the maximum of the prediction error $E [X - E[X|X+Y]]^2$ is strictly increasing as the variance of the noise $Y$ increases. In other words, ``more noise makes prediction worse.'' This conclusively answers Question 4 posed on page 2 of \cite{10.1214/22-ECP467}.

%The rest of the paper is organized as follows. In Section \ref{sec:existence}, we first prove the proposition regarding the convergence result for the variance of conditional expectations (Proposition \ref{prop:keylemmatoshowexistence}). Relying on this proposition, we successfully show the existence of the most stable ``additive'' which minimizes $\mathrm{var}\, E[X|X+Y]$ in the class $\mathcal{V}_{\epsilon}(X)$. A property of the most stable ``additive'' is derived in Section \ref{sec:property}. More specifically, we show that the variance of the most stable ``additive'' must be $\epsilon^2$. With this property, to search for the most stable ``additive'', we only need restrict our attention on the class of Y with mean zero and variance $\epsilon^2$. In section \ref{sec:LXepsilon}, we focus on the function $L(X, \epsilon)$. We prove $L(X, \epsilon)$ is strictly decreasing and right continuous with respect to $\epsilon$ on $[0, \infty)$.

\section{Existence of the least favorable noise} \label{sec:existence}
\indent The main result of this section is the proof of the existence of the least favorable noise, which we give in Theorem \ref{thm:existencetheorem}.\\
\indent We begin by recording Lemma \ref{lemma:usefulleemaformomentconvergence} below.
Since this lemma is a direct result of Skorokhod's representation theorem, Fatou's lemma, and Theorem 4.5.2 in \cite{chung2001course}, we omit the proof.
\begin{lemma} \label{lemma:usefulleemaformomentconvergence}
Suppose $\{X_{n}\}_{n=1}^{\infty}$ is a sequence of random variables converging weakly to some random variable $X$. Further, assume that
\begin{equation*}
\sup_{1 \leq n < \infty} E[X_n^2] < \infty.
\end{equation*}
Then
\begin{equation*}
E[X^2] \leq \liminf_{n \rightarrow \infty} E[X_n^2] \ \text{ and } \ 
\lim_{n \rightarrow \infty} E[X_n] = E[X].
\end{equation*}
\end{lemma}

We proceed to introduce Theorem \ref{prop:keylemmatoshowexistence} below, which gives a convergence result for the variances of conditional expectations. It is, in part, inspired by the work of \cite{crimaldi2005convergence}. 
\begin{theorem} \label{prop:keylemmatoshowexistence}
Let $\{(X_n, Y_n)\}_{n=1}^{\infty}$ be a sequence of random vectors which converges weakly to some random vector $(X,Y)$. If
\begin{equation*}
 \sup_{1 \leq n < \infty} E[X_n^2] < \infty,
\end{equation*}
then
\begin{equation*}
\mathrm{var} \, E[X|X+Y] \leq \liminf_{n \rightarrow \infty} \, \mathrm{var} \, E[X_n|X_n+Y_n].
\end{equation*}
\end{theorem}

\begin{proof}
We denote $\sup E[X_n^2]$ by $M$. By the Continuous Mapping Theorem, $X_n$ converges weakly to $X$. Applying Lemma \ref{lemma:usefulleemaformomentconvergence} yields 

\begin{equation}
E[X^2] \leq \liminf E[X_n^2] \leq M.
\end{equation}

\begin{equation} \label{eq:limexn=ex}
\lim_{n \rightarrow \infty} E[X_n] = E[X].
\end{equation}
By standard properties of conditional expectation, the mean of $E[X_n | X_n + Y_n]$ is $E[X_n]$. Then 
\begin{equation*}
\mathrm{var} \, E[X_n|X_n+Y_n] = E\left[ \left( E[X_n| X_n+Y_n] \right)^{2} \right] - \left(E[X_n]\right)^2. 
\end{equation*}
Similarly, $\mathrm{var} \, E[X|X+Y] = E\left[ \left( E[X| X+Y] \right)^{2} \right] - \left(E[X]\right)^2$. Together with \eqref{eq:limexn=ex}, we need only show that
\begin{equation}
 E\left[ \left( E[X| X+Y] \right)^{2} \right] \leq \liminf_{n \rightarrow \infty} 
 E\left[ \left( E[X_n| X_n+Y_n] \right)^{2} \right]. \label{eq:lessthenliminf}
\end{equation}
Since $(X_n , Y_n)$ converges weakly to $(X,Y)$, by Skorohod's representation theorem, we can construct $(U_n, W_n)$ and $(U,W)$ on a new probability space such that $(U_n, W_n) \overset{d}{=} (X_n, Y_n)$, $(U, W) \overset{d}{=} (X, Y)$ and $(U_n, W_n)$ converges to $(U,W)$ almost surely. It follows immediately that
\begin{equation*}
\sup_{1 \leq n < \infty} E[U_n^2] =  \sup_{1 \leq n < \infty} E[X_n^2] =M,
\end{equation*}
and
\begin{equation*}
E[U^2] = E[X^2] \leq M.
\end{equation*}
Noting that the distributions of $E[X_n|X_n+Y_n]$ and $E[U_n|U_n+W_n]$ depend only on the distributions, respectively, of $(X_n, Y_n)$ and $(U_n, W_n)$, it follows from the fact 
\begin{equation*}
(X_n, Y_{n}) \overset{d}{=}  (U_n, W_n)
\end{equation*}
that the distribution of $E[X_n|X_n+Y_n]$ coincides with that of $E[U_n|U_n+W_n]$. Thus,
\begin{equation*}
E\left[ \left( E[X_n| X_n+Y_n] \right)^{2} \right]  = E\left[ \left( E[U_n| U_n+W_n] \right)^{2} \right].
\end{equation*}
Similarly,
\begin{equation*}
E\left[ \left( E[X| X+Y] \right)^{2} \right]  = E\left[ \left( E[U| U+W] \right)^{2} \right].
\end{equation*}
Thus, to show \eqref{eq:lessthenliminf}, it suffices to prove
\begin{equation}
E\left[ \left( E[U| U+W] \right)^{2} \right] \leq \liminf_{n \rightarrow \infty}  E\left[ \left( E[U_n| U_n+W_n] \right)^{2} \right]. \label{eq:finalgoallessliminf}
\end{equation}

\noindent We proceed to define
\begin{equation}
T_n := E[U_n | U_n + W_n] \ \text{ and } \  T := E[U|U+W].  \label{eq:definitionofTandTn}
\end{equation}
By Jensen's inequality, for all $n \in \mathds{N}_{+}$,
\begin{equation*}
E[T_{n}^{2}] = E\left[ \left( E[U_n| U_n+W_n] \right)^{2} \right]
\leq E\left[  E[U_n^2| U_n+W_n]  \right]
= E[U_{n}^{2}] \leq M.
\end{equation*}
Similarly, $E[T^2] \leq M$. 

The proof continues by invoking the following lemma, whose proof is relegated to the Appendix.
\begin{lemma} \label{lm:connectionbetweensequenceandlimit}
For every bounded Borel function $h$ on $\mathbb{R}$,
\begin{equation}
E[h(U+W) \, T] = \lim_{n \rightarrow \infty} E[h(U+W) \, T_n]. \label{eq:limexpectationequal}
\end{equation}
\end{lemma}
\medskip

We now show how Lemma \ref{lm:connectionbetweensequenceandlimit} implies \eqref{eq:finalgoallessliminf}. Let $g$ be a Borel function on $\mathbb{R}$. By the definition of conditional expectation, $E[U|U+W]$ can be represented as $g(U+W)$. For $k \in \mathds{N}_{+}$, let $g_{k} := g \, \mathds{1}_{\{|g| \leq k\}}$. Then $g_{k}$ is a bounded Borel function. Applying Lemma \ref{lm:connectionbetweensequenceandlimit} gives
\begin{equation*}
E[g_{k}(U+W)\, T] = \lim_{n \rightarrow \infty} E[g_{k}(U+W)\, T_{n}].
\end{equation*}
Since 
\begin{equation*}
E[g_{k}(U+W)\, T_{n}] \leq \left( E\left[(g_{k}(U+W))^2\right] + E[T_{n}^{2}]\right) /2,
\end{equation*}
letting $n \rightarrow \infty$ yields
\begin{equation}
E[g_{k}(U+W)\, T] \leq \frac{1}{2} \, \liminf_{n \rightarrow \infty}   E[T_{n}^{2}]  +  \frac{1}{2} \,  E\left[(g_{k}(U+W))^2\right] . \label{eq:arigoemetryinequality}
\end{equation}
Noting that 
\begin{equation*}
g(U+W) = E[U|U+W] =T
\end{equation*}
is square-integrable, it then follows by dominated convergence that
\begin{equation}
\lim_{k \rightarrow \infty} E\left[(g_{k}(U+W))^2\right] = E\left[(g(U+W))^2\right] = E[T^2]. \label{eq:monotomectgetlimit}
\end{equation}
Further, noting that 
\begin{equation*}
|g_{k}(U+W) \,T| \leq |g(U+W) \,T| = T^2,
\end{equation*}
and recalling that $T^2$ is integrable, the family of random variables $\{g_{k}(U+W)\,T\}$ is uniformly integrable. Since $g_{k}(U+W)\,T$ converges to $g(U+W)\,T$ almost surely, we have that
\begin{equation}
\lim_{k \rightarrow \infty} E[g_{k}(U+W)\, T] 
= E[g(U+W)\, T] = E[T^2]. \label{eq:unformlyintegrablegetlimit}
\end{equation}
Finally, letting $k \rightarrow \infty$ on both sides of \eqref{eq:arigoemetryinequality}, and combining the results in \eqref{eq:monotomectgetlimit} and \eqref{eq:unformlyintegrablegetlimit}, we have
\begin{equation*}
E[T^2] \leq \liminf_{n} E[T_{n}^2],
\end{equation*}
which is precisely what appears in \eqref{eq:finalgoallessliminf}. This completes the proof.
\end{proof}

With Theorem \ref{prop:keylemmatoshowexistence} in hand, we are now ready to present the key theorem of this manuscript regarding existence of the least favorable noise.

\begin{theorem} \label{thm:existencetheorem}
Suppose $X$ is a non-degenerate random variable with zero mean and finite second moment. Then there exists a minimizer $Y \in \mathcal{V}_{\epsilon}(X)$ such that
\begin{equation*}
\mathrm{var}\, E[X|X+Y] = \inf_{Z \in \mathcal{V}_{\epsilon}(X)} \mathrm{var}\, E[X|X+Z] = L(X, \epsilon).
\end{equation*}
Consequently, the least favorable noise exists.
\end{theorem}

\begin{proof}
Let $Y_n \in \mathcal{V}_{\epsilon}(X)$ be a random variable such that
\begin{equation*}
\mathrm{var}\, E[X|X+Y_n] \leq  L(X, \epsilon) + \frac{1}{n}.
\end{equation*}
Since $E[X^2] < \infty$ and $\sup_{n} E[Y_n^2] \leq \epsilon^2$, the sequence of random vectors $\{(X,Y_n)\}_{n=1}^{\infty}$ is tight. Thus, there exists a subsequence $\{(X, Y_{n(k)})\}$ of $\{(X,Y_n)\}$ such that $(X, Y_{n(k)})$ converges weakly to some random vector, say, $(\tilde{X}, \tilde{Y})$. It follows by the Continuous Mapping Theorem that $X \overset{d}{=} \tilde{X}$ and $Y_{n(k)}$ converges weakly to $\tilde{Y}$. Noting that $X$ is independent of $Y_n$, it may be easily verified that $\tilde{X}$ is independent of $\tilde{Y}$. Furthermore, applying Lemma \ref{lemma:usefulleemaformomentconvergence} to the sequence $\{Y_{n(k)}\}$, we have
\begin{equation*}
E[\tilde{Y}^2] \leq \liminf_{k \rightarrow \infty} E[Y_{n(k)}^{2}] \leq \epsilon^2 \ \text{ and } \ 
E[\tilde{Y}] = \lim_{k \rightarrow \infty} E[Y_{n(k)}]=0.
\end{equation*}

We now may construct a random variable $Y$ such that $Y$ is independent of $X$ and $Y$ has the same distribution as $\tilde{Y}$. We now claim $Y$ is the desired minimizer. We proceed to prove this claim. Indeed, since 
\begin{equation*}
E[Y]= E[\tilde{Y}] =0
\end{equation*}
and 
\begin{equation*}
E[Y^2] = E[\tilde{Y}^2] \leq \epsilon^2,
\end{equation*}
then $Y \in \mathcal{V}_{\epsilon}(X)$. Furthermore, it follows by the assumption of independence of $X$ and $Y$ that $(X,Y)  \overset{d}{=} (\tilde{X}, \tilde{Y})$. Invoking a similar argument from the proof of Theorem \ref{prop:keylemmatoshowexistence} yields that
\begin{equation*}
\mathrm{var} E[X| X+Y] = \mathrm{var} E[\tilde{X}| \tilde{X} + \tilde{Y}].
\end{equation*}
 Finally, applying Theorem \ref{prop:keylemmatoshowexistence} yields
\begin{equation*}
\mathrm{var} E[\tilde{X}| \tilde{X} + \tilde{Y}] \leq
\liminf_{k \rightarrow \infty} \mathrm{var} E[X_{n(k)}| X_{n(k)}+Y_{n(k)}] = L(X, \epsilon).
\end{equation*}
Thus, $\mathrm{var} E[X| X+Y] \leq L(X, \epsilon)$, which proves that $Y$ is the desired minimizer.
\end{proof}

\section{Variance of the least favorable noise} \label{sec:property}
The purpose of this section is to calculate the variance of the least favorable noise (Theorem \ref{prop:variancemustequal}). Theorem \ref{prop:variancemustequal} relies on Proposition \ref{lm:inequalityforsecondmomentofconex} and Lemma \ref{lm:strictlyless} below. Proposition  \ref{lm:inequalityforsecondmomentofconex} is a standard result conditional expectation and therefore we omit the proof.
\begin{proposition} \label{lm:inequalityforsecondmomentofconex}
Suppose $X$ is a square-integral random variable. Let $\mathcal{F}$ and $\mathcal{G}$ be two $\sigma$-algebras with $\mathcal{G} \subset \mathcal{F}$. Then
\begin{equation*}
E\left[ \left( E[X | \mathcal{G}] \right)^2 \right] \leq E\left[ \left( E[X | \mathcal{F}] \right)^2 \right],
\end{equation*}
with equality holding if and only if $E[X | \mathcal{G}] = E[X | \mathcal{F}]$ almost surely.
\end{proposition}
We now introduce Lemma \ref{lm:strictlyless}, which relies on Proposition \ref{lm:inequalityforsecondmomentofconex}.
\begin{lemma} \label{lm:strictlyless}
Suppose $X$ is a non-degenerate random variable with zero mean and finite second moment and $Y$ is an arbitrary random variable independent of $X$. Let $Z$ be a non-degenerate Gaussian random variable independent of $\sigma(X,Y)$. Then
\begin{equation*}
\mathrm{var}\, E[X|X+Y+Z] < \mathrm{var}\, E[X|X+Y].
\end{equation*}
\end{lemma}
\begin{proof}
First, we note that the means of $E[X|X+Y]$ and $E[X|X+Y+Z]$ are both $E[X]=0$. It thus suffices to prove that
\begin{equation*}
E\left[ \left( E[X | X+Y+Z] \right)^2 \right] < E\left[ \left( E[X | X+Y] \right)^2 \right].
\end{equation*}
Applying Proposition \ref{lm:inequalityforsecondmomentofconex} gives
\begin{equation}
E\left[ \left( E[X | X+Y+Z] \right)^2 \right] \leq E\left[ \left( E[X | X+Y, Z] \right)^2 \right] = E\left[ \left( E[X | X+Y] \right)^2 \right], \label{eq:secondmomentsmall}
\end{equation}
where in the last equality we have invoked the fact that
\begin{equation*}
E[X| X+Y, Z]= E[X|X+Y], 
\end{equation*}
which holds because $Z$ is independent of $\sigma(X,Y)$. In what follows, we only need rule out the case where
\begin{equation}
E\left[ \left( E[X | X+Y+Z] \right)^2 \right] = E\left[ \left( E[X | X+Y] \right)^2 \right] \label{eq:twovarianceequal}.
\end{equation}

We proceed by contradiction. Assume, for the sake of contradiction, that equation \eqref{eq:twovarianceequal} holds. Combining \eqref{eq:secondmomentsmall} and Proposition \ref{lm:inequalityforsecondmomentofconex}, we have that
\begin{equation}
E[X | X+Y+Z] = E[X | X+Y, Z] =E[X | X+Y], \label{eq:xyz=xy}
\end{equation}
holds almost surely.\\
\indent Let $f_1$ and $f_2$ be two Borel functions on $\mathbb{R}$. Let $f_{1}(X+Y+Z)$ and $f_{2}(X+Y)$ be versions of $E[X|X+Y+Z]$ and $E[X|X+Y]$ respectively. Then the equality in \eqref{eq:xyz=xy} implies that
\begin{equation}
P\left( f_{1}(X+Y+Z) = f_{2}(X+Y) \right) =1.  \label{eq:fxyz=xy}
\end{equation}
In what follows, we use $P_{U}$ to denote the probability measure on $\mathbb{R}$ generated by the random variable $U$. Since $X,Y,Z$ are independent, \eqref{eq:fxyz=xy} is equivalent to
\begin{equation*}
\int_{\mathbb{R}} \int_{\mathbb{R}}  P(f_{1}(x+y+ Z) = f_{2}(x+y)) \,P_{X}(dx) \, P_{Y}(dy) = 1.
\end{equation*}
Since $P(f_{1}(x+y+ Z) = f_{2}(x+y)) \leq 1$, we have that
\begin{equation*}
P(f_{1}(x+y+ Z) = f_{2}(x+y)) =1 \quad \quad \text{$P_{X} \otimes P_{Y} $ - a.s..}
\end{equation*}
Then we can select $x_0$, $y_0$ such that
\begin{equation*}
P(f_{1}(x_0+y_0+ Z) = f_{2}(x_0+y_0)) =1,
\end{equation*}
which implies that $f_{1}(x_0 +y_0 +Z)$ is a constant almost surely. Note that since $Z$ is an absolutely continuous random variable whose density is positive everywhere, $f_1$ is a constant almost everywhere with respect to Lebesgue measure. By standard properties of convolution, $X+Y+Z$ is an absolutely continuous random variable. Thus, $f_{1}(X+Y+Z)$ is a constant almost surely, namely, $E[X|X+Y+Z]$ is a constant almost surely. Note that since the mean of $E[X|X+Y+Z]$ is $0$, $E[X|X+Y+Z]$ must be $0$ almost surely. Then
\begin{eqnarray*}
0 &=& E\left[ E[X|X+Y+Z] (X+Y+Z) \right]
= E\left[ X(X+Y+Z) \right]  \\
&=& E[X^2] + E[XY] +E[XZ] = E[X^2] + E[X]E[Y] + E[X]E[Z] = E[X^2],
\end{eqnarray*}
which implies that $X=0$ almost surely. Since $X$ is non-degenerate, a contradiction has been reached. This concludes the proof.
\end{proof}

With above lemmas in hand, we now turn to the variance of the least favorable noise.

\begin{theorem} \label{prop:variancemustequal}
Suppose $X$ is a non-degenerate random variable with zero mean and finite second moment. For any $Y \in \mathcal{V}_{\epsilon}(X)$ such that $\mathrm{var}\, E[X|X+Y]$ attains the minimum $L(X, \epsilon)$, we must have that
$E[Y^2] = \epsilon^2$.
\end{theorem}
\begin{proof}
We proceed by contradiction. For the sake of contradiction, let us assume $E[Y^2] < \epsilon^2$. We proceed to construct a Gaussian random variable $Z$ with zero mean and sufficiently small variance such that $Z$ is independent of $\sigma(X,Y)$ and $E[(Y+Z)^2] \leq \epsilon^2$. It is immediate to verify that $Y+Z \in \mathcal{V}_{\epsilon}(X)$. However, applying Lemma \ref{lm:strictlyless} gives
\begin{equation*}
\mathrm{var}\, E[X|X+Y+Z] < \mathrm{var}\, E[X|X+Y].
\end{equation*}
However, this directly contradicts the fact that $Y$ is a minimizer. Thus, by contradiction,  $E[Y^2] = \epsilon^2$. 
\end{proof}

\begin{remark}
Theorem \ref{prop:variancemustequal} tells us that in order to search for the least favorable noise, one may only need focus on the random variable $Y$ satisfying $E[Y^2] = \epsilon^2$. If one attempts to find the least favorable noise by the method of Lagrange multipliers, Theorem \ref{prop:variancemustequal} converts the inequality constraint to an equality constraint, which greatly simplifies the problem.
\end{remark}

\section{Properties of $L(X, \epsilon)$} \label{sec:LXepsilon}
The purpose of this section is to study some properties of the function $L(X, \epsilon)$. Proposition \ref{prop:propertyoneofL} is introduced in order to prove the key result of this section (Theorem \ref{prop:rightcontinuityforL}), which states that $L(X, \epsilon)$ is strictly decreasing and right continuous in $\epsilon$ on $[0, \infty)$.

\begin{proposition} \label{prop:propertyoneofL}
Let $\{ X_n \}_{n=1}^{\infty}$ be a sequence of random variables with zero mean which converges weakly to some random variable $X$ with zero mean. Let $\{\epsilon_n\}_{n=1}^{\infty}$ be a sequence of non-negative real number which converges to a non-negative real number $\epsilon$. If
\begin{equation*}
\sup_{1 \leq n < \infty} E[X_n^2] < \infty,
\end{equation*}
then
\begin{equation*}
L(X, \epsilon) \leq \liminf_{n \rightarrow \infty} L(X_n, \epsilon_n).
\end{equation*}
\end{proposition}

\begin{proof}
By the properties of limit inferior, there exists a subsequence $\{n(k)\}_{k=1}^{\infty}$ of $\{n\}_{n=1}^{\infty}$ such that
\begin{equation*}
\liminf_{n \rightarrow \infty} L(X_n, \epsilon_n) = \lim_{k \rightarrow \infty} L(X_{n(k)}, \epsilon_{n(k)}).
\end{equation*}
Thus, provided the limit of $ L(X_{n(k)}, \epsilon_{n(k)}) $ exists, we need only prove
\begin{equation*}
L(X, \epsilon) \leq \liminf_{k \rightarrow \infty} L(X_{n(k)}, \epsilon_{n(k)}).
\end{equation*}

\noindent For simplicity of notation, in what follows, we continue to write the subsequence $\{n(k)\}_{k=1}^{\infty}$ by $\{n\}_{n=1}^{\infty}$. We shall also assume the limit of $L(X_n, \epsilon_n)$ exists.

For every $n$, by Theorem \ref{thm:existencetheorem}, there exists a $Y_n \in \mathcal{V}_{\epsilon_n} (X_n)$ such that
\begin{equation*}
\mathrm{var} E[X_n| X_n+Y_n] = L(X_n, \epsilon_n).
\end{equation*}
Noting that $E[Y_{n}^{2}] \leq \epsilon_{n}^{2}$ and that $\epsilon_n$ converges to $\epsilon$, we have $\sup_{n} E[Y_{n}^{2}] < \infty$. Together with the fact that $\sup_{n} E[X_{n}^{2}] < \infty$, the family of random vectors $\{(X_n, Y_n)\}_{n=1}^{\infty}$ is tight. Then there exists a subsequence $\{(X_{n(k)}, Y_{n(k)})\}$ of $\{(X_n, Y_n)\}$ such that $(X_{n(k)}, Y_{n(k)})$ converges weakly to some random vector, say, $(\tilde{X}, \tilde{Y})$. By continuous mapping, we have that $X_{n(k)}$ converges weakly to $\tilde{X}$ and $Y_{n(k)}$ converges weakly to $\tilde{Y}$. Then, $X \overset{d}{=} \tilde{X}$, since $X_{n(k)}$ also converges weakly to $X$. By Lemma \ref{lemma:usefulleemaformomentconvergence}, we have
$E[\tilde{Y}] = \lim_{k} E[Y_{n(k)}] = 0$ and $E[\tilde{Y}^2] \leq \liminf_{k} E[Y_{n(k)}^2] \leq \epsilon^2$. Furthermore, since $X_{n(k)}$ is independent of $Y_{n(k)}$, we also have that $\tilde{X}$ is independent of $\tilde{Y}$. Invoking Theorem \ref{prop:keylemmatoshowexistence} here gives
\begin{eqnarray}
&&\mathrm{var} E[\tilde{X}| \tilde{X} + \tilde{Y}] \leq \liminf_{k \rightarrow \infty} \, \mathrm{var} E[X_{n(k)}| X_{n(k)}+Y_{n(k)}] \notag \\
&=& \liminf_{k \rightarrow \infty} \,  L(X_{n(k)}, \epsilon_{n(k)})
= \liminf_{n \rightarrow \infty}\, L(X_n, \epsilon_n), \label{eq:propertyoneofLlessthan}
\end{eqnarray}
where the last equality follows by our assumption that the limit of $L(X_n, \epsilon_n)$ exists.

There exists a random variable $Y$ such that $Y$ is independent of $X$ and $Y \overset{d}{=} \tilde{Y}$. Thus, $E[Y]= E[\tilde{Y}] =0$ and $E[Y^2] = E[\tilde{Y}^2] \leq \epsilon^2$, which implies $Y \in \mathcal{V}_{\epsilon}(X)$. By independence, we have $(X,Y) \overset{d}{=} (\tilde{X}, \tilde{Y})$, hence,
\begin{equation}
\mathrm{var} E[X|X+Y] = 
 \mathrm{var} E[\tilde{X}| \tilde{X} + \tilde{Y}], \label{eq:propertyoneofLequalvar}
\end{equation}
by a similar argument in the proof of Theorem \ref{prop:keylemmatoshowexistence}. Combining \eqref{eq:propertyoneofLlessthan} and \eqref{eq:propertyoneofLequalvar}, we have
\begin{equation*}
L(X, \epsilon) \leq \mathrm{var} E[X|X+Y] \leq \liminf_{n \rightarrow \infty}\, L(X_n, \epsilon_n).
\end{equation*} 
This completes the proof.
\end{proof}

\begin{theorem} \label{prop:rightcontinuityforL}
Let $X$ be a non-degenerate random variable with zero mean and finite second moment. Then, with fixed $X$, $L(X, \epsilon)$ is a strictly decreasing and right continuous function with respect to $\epsilon$ on $[0, \infty)$.
\end{theorem}

\begin{proof}
We first shall prove that $L(X, \epsilon)$ is strictly decreasing with respect to $\epsilon$. Consider $0 \leq \epsilon_1 < \epsilon_2$. By Theorem \ref{thm:existencetheorem}, there exists $Y_1 \in \mathcal{V}_{\epsilon_1}(X)$ such that
\begin{equation*}
\mathrm{var} \, E[X|X+Y_1] = L(X, \epsilon_1).
\end{equation*}
We proceed to construct a Gaussian random variable with mean zero and variance $\epsilon_2^2 - \epsilon_1^2$ such that $Z$ is independent of $\sigma(X,Y_1)$. It is straightforward to check that $Y_1 +Z \in \mathcal{V}_{\epsilon_2}(X)$. Then by Lemma \ref{lm:strictlyless},
\begin{equation*}
L(X, \epsilon_2) \leq \mathrm{var} \, E[X|X+Y_1+Z] < \mathrm{var} \, E[X|X+Y_1] = L(X, \epsilon_1),
\end{equation*}
which means that $L(X, \epsilon)$ is strictly decreasing.

We now turn to the right continuity. Consider every $\epsilon_0 \in [0, \infty)$. For $\epsilon > \epsilon_{0}$, we have $L(X, \epsilon) < L(X, \epsilon_{0})$. Thus
\begin{equation}
\limsup_{\epsilon \downarrow \epsilon_{0}} L(X, \epsilon) \leq L(X, \epsilon_{0}). \label{eq:rightcontinuousforL1}
\end{equation}
Applying Proposition \ref{prop:propertyoneofL} yields
\begin{equation}
L(X, \epsilon_0) \leq \liminf_{\epsilon \downarrow \epsilon_{0}} L(X, \epsilon).  \label{eq:rightcontinuousforL2}
\end{equation}
Finally, combining \eqref{eq:rightcontinuousforL1} and \eqref{eq:rightcontinuousforL2} completes the proof.
\end{proof}

\begin{remark}
Theorem \ref{prop:rightcontinuityforL} shows that the maximum of the prediction error $E [X - E[X|X+Y]]^2$ is strictly increasing as the variance of the noise $Y$ increases. In other words, ``more noise makes prediction worse.'' This conclusively answers Question 4 posed on page 2 of \cite{10.1214/22-ECP467}.
\end{remark}
\bigskip

\noindent \textbf{Acknowledgements:} I thank my former Ph.D. advisor, Professor Philip A. Ernst, for many helpful and insightful discussions about this paper.

\bibliographystyle{plain}
\bibliography{reference}

\newpage

\section{Appendix}
This Appendix contains the proof of Lemma \ref{lm:connectionbetweensequenceandlimit}.
\begin{proof}[Proof of Lemma \ref{lm:connectionbetweensequenceandlimit}]
We first claim that the class $\mathcal{H}$ of bounded real-valued Borel function $h$ on $\mathbb{R}$ such that \eqref{eq:limexpectationequal} holds must satisfy properties (A.1)-(A.3) below:
\begin{enumerate}[label={(A.\arabic*)}]
  \item $\mathcal{H}$ is a vector space which contains constant functions;
  \item $\mathcal{H}$ is closed under uniform convergence;
  \item For a uniformly bounded sequence $\{h_k\}$ of non-negative functions in $\mathcal{H}$ where $\forall k$, $h_{k} \leq h_{k+1}$, and  $\forall s$, $h_{k}(s)  \rightarrow h(s)$, we have that $h \in \mathcal{H}$.
\end{enumerate}
It is immediate that $\mathcal{H}$ is a vector space. Furthermore, for every constant function $c(x) \equiv c $, by \eqref{eq:limexn=ex}, we have
\begin{equation*}
\lim_{n \rightarrow \infty} E[c(U+W) \,T_n] = \lim_{n \rightarrow \infty} c \, E[T_n] = \lim_{n \rightarrow \infty} c \, E[X_n] =
 c\, E[X] = c\, E[T] = E[c(U+W) \,T].
\end{equation*}
Thus, property (A.1) holds.

To check property (A.2), suppose $h_{k} \in \mathcal{H}$ converges uniformly to $h$. Then
\begin{eqnarray*}
&&|E[h(U+W)\, T] - E[h(U+W)\,  T_n]|  \\
&\leq&  \left|E\left[\left(h(U+W) - h_{k}(U+W)\right) \, T \right]\right| + \left|E\left[\left(h(U+W) - h_{k}(U+W)\right) \, T_n \right]\right| \\
&& + |E[h_{k}(U+W) \, T] - E[h_{k}(U+W) \, T_n]| \\
&\leq& \left\{ E\left[ \left(h(U+W) - h_{k}(U+W)\right)^2 \right] \right\}^{1/2} \cdot \left\{E[T^2]\right\}^{1/2}  \\
&& + \left\{ E\left[ \left(h(U+W) - h_{k}(U+W)\right)^2 \right] \right\}^{1/2} \cdot \left\{E[T_n^2]\right\}^{1/2}  \\
&& + |E[h_{k}(U+W) \, T] - E[h_{k}(U+W) \, T_n]| \\
&\leq& 2M^{1/2} \, \left\{ E\left[ \left(h(U+W) - h_{k}(U+W)\right)^2 \right] \right\}^{1/2}  \\
&& + |E[h_{k}(U+W) \, T] - E[h_{k}(U+W) \, T_n]|,
\end{eqnarray*}
where the second inequality follows by H\"{o}lder's inequality and in the last inequality we have applied the fact that $\sup_{n} E[T_n^2] \leq M$ and $E[T^2] \leq M$.
Letting $n$ tend to $\infty$ and noting that \eqref{eq:limexpectationequal} holds for $h_k$, we have
\begin{equation}
\limsup_{n\rightarrow \infty} |E[h(U+W)T] - E[h(U+W) T_n]|
\leq 2M^{1/2} \, \left\{ E\left[ \left(h(U+W) - h_{k}(U+W)\right)^2 \right] \right\}^{1/2}. \label{eq:holderbound}
\end{equation}
Letting $k$ tend to $\infty$, and recalling the fact that $h_k$ converges uniformly to $h$, we conclude that $h \in \mathcal{H}$.

We now turn to (A.3). Suppose uniformly bounded non-negative functions $h_k\uparrow h$. It is immediate that $h_{k}(U+W) - h(U+W)$ converges to $0$ pointwise and that it is uniformly bounded. By dominated convergence, we have
\begin{equation}
\lim_{n \rightarrow \infty} E\left[ \left(h(U+W) - h_{k}(U+W)\right)^2 \right] = 0 \label{eq:senondmomentto0forh}
\end{equation}
Using a similar argument in the proof of (A.2) yields \eqref{eq:holderbound} again. Together with \eqref{eq:senondmomentto0forh} we conclude that $h \in \mathcal{H}$.

With above preparation in hand, we prove the equality in \eqref{eq:limexpectationequal} for every bounded Borel function $h$. By the monotone class theorem (cf. \cite[p.91]{rogers2000diffusions}), it suffices to prove  \eqref{eq:limexpectationequal} for every bounded continuous function $h$ on $\mathbb{R}$. Recall that by construction, $(U_n, W_n)$ converges to $(U,W)$ almost surely. Since $h$ is continuous, we have $h(U_n+W_n) \rightarrow h(U+W)$ almost surely. Since $h$ is bounded, by dominated convergence, we have
\begin{equation}
\lim_{n \rightarrow \infty}E\left[ \left( h(U_n + W_n) - h(U+W)\right)^{2} \right] = 0. \label{eq:hofUnclosetohofU}
\end{equation}
Note that
\begin{equation*}
\sup_{1\leq n < \infty} E\left[ \left( U_{n} h(U_n + W_n) \right)^{2}\right] \leq \|h\|^2 \cdot \sup_{1 \leq n < \infty} E[U_{n}^{2}] = M \|h\|^2,
\end{equation*}
where $\|h\| := \sup_{x} |h(x)|$. Then the family of random variables $\{U_{n} h(U_n + W_n)\}_{n=1}^{\infty}$ is uniformly integrable. Further, $U_{n} h(U_n + W_n)$ converges to $U h(U+W)$ almost surely, and so
\begin{equation}
\lim_{n \rightarrow \infty} E[U_{n} h(U_n + W_n)] = E[U h(U + W)]. \label{eq:uhuwntouhuw}
\end{equation}
Recalling the definitions of $T_n$ and $T$ given in \eqref{eq:definitionofTandTn}, and invoking the standard properties of conditional expectation, we have
\begin{equation*}
E[h(U_n + W_n) \, T_n] = E\left[h(U_n + W_n) E[U_{n} | U_n + W_n ]\right] = E[U_{n} h(U_n + W_n)],
\end{equation*}
and
\begin{equation*}
E[h(U + W) \, T] = E\left[h(U + W) E[U | U + W ]\right] = E[U h(U + W)].
\end{equation*}
Plugging the above two displays into \eqref{eq:uhuwntouhuw} yields
\begin{equation}
\lim_{n \rightarrow \infty} E[h(U_n + W_n) \, T_n] = E[h(U + W) \, T]. \label{eq:onelink}
\end{equation}
By H\"{o}lder's inequality, we have
\begin{eqnarray*}
&&\left| E[h(U_n + W_n) \, T_n] - E[h(U + W) \, T_n]\right| \\
 &=& \left| E[(h(U_n + W_n)- h(U + W)) \, T_n]  \right| \\
 &\leq& \left\{ E\left[ \left(h(U_n+W_n) - h(U+W)\right)^2 \right] \right\}^{1/2} \cdot \left\{ E[T_{n}^{2}] \right\}^{1/2} \\
 &\leq& M^{1/2} \, \left\{ E\left[ \left(h(U_n+W_n) - h(U+W)\right)^2 \right] \right\}^{1/2}.
\end{eqnarray*}
Together with \eqref{eq:hofUnclosetohofU}, we obtain
\begin{equation}
\lim_{n\rightarrow \infty} \left| E[h(U_n + W_n) \, T_n] - E[h(U + W) \, T_n]\right| = 0 \label{eq:twolink}
\end{equation}
Combining \eqref{eq:onelink} and \eqref{eq:twolink}, the equality in \eqref{eq:limexpectationequal} follows.
\end{proof}

\newpage

\end{document}